\documentclass{amsart}
\usepackage{setspace}
\usepackage{a4}
\usepackage{amsthm}
\usepackage{latexsym}
\usepackage{amsfonts}
\usepackage{graphicx}
\usepackage{textcomp}
\usepackage{cite}
\usepackage{enumerate}
\usepackage{amssymb}
\usepackage{hyperref}
\usepackage{amsmath}
\usepackage{tikz}
\usepackage[mathscr]{euscript}
\usepackage{mathtools}
\newtheorem{theorem}{Theorem}[section]

\newtheorem{conjecture}[theorem]{Conjecture}
\newtheorem{corollary}[theorem] {Corollary}
\newtheorem{definition}[theorem]{Definition}
\newtheorem{example}[theorem]{Example}

\newtheorem{problem}[theorem]{Problem}

\newtheorem{question}[theorem]{Question}
\title{This is the title}
\usepackage{fancyhdr}

\begin{document}
\hrule\hrule\hrule\hrule\hrule
\vspace{0.3cm}	
\begin{center}
{\bf{FUNCTIONAL CONTINUOUS  UNCERTAINTY PRINCIPLE}}\\
\vspace{0.3cm}
\hrule\hrule\hrule\hrule\hrule
\vspace{0.3cm}
\textbf{K. MAHESH KRISHNA}\\
Post Doctoral Fellow \\
Statistics and Mathematics Unit\\
Indian Statistical Institute, Bangalore Centre\\
Karnataka 560 059, India\\
Email: kmaheshak@gmail.com\\

Date: \today
\end{center}

\hrule\hrule
\vspace{0.5cm}
\textbf{Abstract}: Let $(\Omega, \mu)$,  $(\Delta, \nu)$ be   measure spaces. Let  $(\{f_\alpha\}_{\alpha\in \Omega}, \{\tau_\alpha\}_{\alpha\in \Omega})$  and   $(\{g_\beta\}_{\beta\in \Delta}, \{\omega_\beta\}_{\beta\in \Delta})$   be	continuous p-Schauder frames  for a Banach space $\mathcal{X}$. Then for every $x \in \mathcal{X}\setminus\{0\}$,  we show that 
\begin{align}\label{CUE}
	\mu(\operatorname{supp}(\theta_f x))^\frac{1}{p}	\nu(\operatorname{supp}(\theta_g x))^\frac{1}{q} \geq 	\frac{1}{\displaystyle\sup_{\alpha \in \Omega, \beta \in \Delta}|f_\alpha(\omega_\beta)|}, \quad  	\nu(\operatorname{supp}(\theta_g x))^\frac{1}{p}	\mu(\operatorname{supp}(\theta_f x))^\frac{1}{q}\geq \frac{1}{\displaystyle\sup_{\alpha \in \Omega , \beta \in \Delta}|g_\beta(\tau_\alpha)|}.
\end{align}
where 
\begin{align*}
	&\theta_f: \mathcal{X} \ni x \mapsto \theta_fx \in \mathcal{L}^p(\Omega, \mu); \quad   \theta_fx: \Omega \ni \alpha \mapsto  (\theta_fx) (\alpha):= f_\alpha (x) \in \mathbb{K},\\
	&\theta_g: \mathcal{X} \ni x \mapsto \theta_gx \in \mathcal{L}^p(\Delta, \nu); \quad   \theta_gx: \Delta \ni \beta \mapsto  (\theta_gx) (\beta):= g_\beta (x) \in \mathbb{K}
\end{align*}
and $q$ is the conjugate index of $p$. We call Inequality (\ref{CUE}) as \textbf{Functional Continuous  Uncertainty Principle}. It improves the Functional Donoho-Stark-Elad-Bruckstein-Ricaud-Torr\'{e}sani Uncertainty Principle obtained by K. Mahesh  Krishna in \textit{[arXiv:2304.03324v1 [math.FA], 5 April 2023]}.   It also answers a question asked by Prof. Philip B. Stark  to the author. Based on Donoho-Elad Sparsity Theorem, we formulate Measure Minimization Conjecture.



\hrule

\tableofcontents
\hrule
\section{Introduction}
For $h \in \mathbb{C}^d$, let $\|h\|_0$ be the number of nonzero entries in $h$. Let $\hat{}: \mathbb{C}^d \to  \mathbb{C}^d$ be the Fourier transform. Great and surprising inequality of Donoho and Stark  which improved our life is the following.
\begin{theorem} (\textbf{Donoho-Stark Uncertainty Principle}) \cite{DONOHOSTARK} \label{DS}
	For every $d\in \mathbb{N}$, 
	\begin{align}\label{DSE}
  \left(\frac{\|h\|_0+\|\widehat{h}\|_0}{2}\right)^2	\geq \|h\|_0\|\widehat{h}\|_0	\geq d, \quad \forall h \in \mathbb{C}^d\setminus \{0\}.
	\end{align}
\end{theorem}
In 2002, Elad and Bruckstein extended  Inequality (\ref{DSE})  to pairs of orthonormal bases \cite{ELADBRUCKSTEIN}.  Given a collection $\{\tau_j\}_{j=1}^n$ in a finite dimensional Hilbert space $\mathcal{H}$ over $\mathbb{K}$ ($\mathbb{R}$ or $\mathbb{C}$), we define 
\begin{align*}
	\theta_\tau: \mathcal{H} \ni h \mapsto \theta_\tau h \coloneqq (\langle h, \tau_j\rangle)_{j=1}^n \in \mathbb{K} ^n.
\end{align*}
\begin{theorem} (\textbf{Elad-Bruckstein Uncertainty Principle}) \cite{ELADBRUCKSTEIN} \label{EB}
	Let $\{\tau_j\}_{j=1}^n$,  $\{\omega_j\}_{j=1}^n$ be two orthonormal bases for a  finite dimensional Hilbert space $\mathcal{H}$. Then 
	\begin{align*}
	\left(\frac{\|\theta_\tau h\|_0+\|\theta_\omega h\|_0}{2}\right)^2	\geq \|\theta_\tau h\|_0\|\theta_\omega h\|_0\geq \frac{1}{\displaystyle\max_{1\leq j, k \leq n}|\langle\tau_j, \omega_k \rangle|^2}, \quad \forall h \in \mathcal{H}\setminus \{0\}.
	\end{align*}
\end{theorem}
 In 2013, Ricaud and Torr\'{e}sani  showed that orthonormal bases in Theorem \ref{EB} can be improved to Parseval frames \cite{RICAUDTORRESANI}.
\begin{theorem} (\textbf{Ricaud-Torr\'{e}sani Uncertainty Principle}) \cite{RICAUDTORRESANI} \label{RT}
	Let $\{\tau_j\}_{j=1}^n$,  $\{\omega_j\}_{j=1}^n$ be two Parseval frames   for a  finite dimensional Hilbert space $\mathcal{H}$. Then 
\begin{align*}
	\left(\frac{\|\theta_\tau h\|_0+\|\theta_\omega h\|_0}{2}\right)^2	\geq \|\theta_\tau h\|_0\|\theta_\omega h\|_0\geq \frac{1}{\displaystyle\max_{1\leq j, k \leq n}|\langle\tau_j, \omega_k \rangle|^2}, \quad \forall h \in \mathcal{H}\setminus \{0\}.
		\end{align*}	
\end{theorem}
In 2023, author made a major improvement of Theorem \ref{RT}.
\begin{theorem} \label{FDS}(\textbf{Functional Donoho-Stark-Elad-Bruckstein-Ricaud-Torr\'{e}sani Uncertainty Principle}) \cite{KRISHNA1}
	Let $(\{f_j\}_{j=1}^n, \{\tau_j\}_{j=1}^n)$ and $(\{g_k\}_{k=1}^m, \{\omega_k\}_{k=1}^m)$ be  p-Schauder frames  for a finite dimensional Banach space $\mathcal{X}$. Then for every $x \in \mathcal{X}\setminus\{0\}$,  we have 
	\begin{align}\label{FDSUNCE}
		\|\theta_f x\|_0^\frac{1}{p}\|\theta_g x\|_0^\frac{1}{q} \geq 	\frac{1}{\displaystyle\max_{1\leq j\leq n, 1\leq k\leq m}|f_j(\omega_k)|}\quad \text{and} \quad \|\theta_g x\|_0^\frac{1}{p}\|\theta_f x\|_0^\frac{1}{q}\geq \frac{1}{\displaystyle\max_{1\leq j\leq n, 1\leq k\leq m}|g_k(\tau_j)|}.
	\end{align}
\end{theorem}
By seeing paper \cite{KRISHNA1},  Prof. Philip B. Stark, Department of Statistics, University of California, Berkeley, asked the following question to the author.
\begin{question} (\textbf{Philip B. Stark})\label{SP}
\textbf{What is the infinite dimensional version of Theorem \ref{FDS}?}
\end{question}
In this paper, we derive continuous uncertainty principle for  Banach spaces which contains Theorem \ref{FDS} as a particular case and also answers Question \ref{SP}.

\section{Functional Continuous Uncertainty Principle}
In the paper,   $\mathbb{K}$ denotes $\mathbb{C}$ or $\mathbb{R}$ and $\mathcal{X}$ denotes a  Banach space (need not be finite dimensional) over $\mathbb{K}$. Dual of $\mathcal{X}$ is denoted by $\mathcal{X}^*$. Whenever $1<p<\infty$, $q$ denotes conjugate index of $p$.  We recall the notion of weak integral also known as Pettis integrals \cite{TALAGRAND}. 	Let 	$(\Omega, \mu)$ be a measure space and  $\mathcal{X}$ be a Banach space. A function $f: \Omega \to \mathcal{X}$ is said to be weak integrable or  Pettis integrable if following conditions hold.
\begin{enumerate}[\upshape(i)]
	\item For every $\phi\in \mathcal{X}^*$, the map $\phi f: \Omega \to \mathbb{K}$  is measurable and $\phi f \in \mathcal{L}^1(\Omega, \mu)$.
	\item For every measurable subset $E\subseteq  \Omega$, there exists an (unique) element $x_E \in \mathcal{X}$ such that
	\begin{align*}
		\phi(x_E)= \int\limits_{E}	\phi(f(\alpha)) \, d \mu(\alpha), \quad \forall \phi\in \mathcal{X}^*.
	\end{align*}
	The element $x_E \in \mathcal{X}$ is denoted by $\int_{E}f(\alpha) \, d \mu(\alpha)$. With this notion, we have
	\begin{align*}
		\phi\left(\int\limits_{E}f(\alpha) \, d \mu(\alpha)\right)= \int\limits_{E}	\phi(f(\alpha)) \, d \mu(\alpha), \quad \forall \phi\in \mathcal{X}^*, \forall 	E\subseteq  \Omega.
	\end{align*}
\end{enumerate}
We need the following continuous version of  p-Schauder frames  \cite{KRISHNA1}.
 \begin{definition}\label{CPSF}
	Let 	$(\Omega, \mu)$ be a measure space. Let    $\{\tau_\alpha\}_{\alpha\in \Omega}$ be a collection in a Banach   space $\mathcal{X}$ and     $\{f_\alpha\}_{\alpha\in \Omega}$ be a collection in  $\mathcal{X}^*$. The pair $(\{f_\alpha\}_{\alpha\in \Omega}, \{\tau_\alpha\}_{\alpha\in \Omega})$   is said to be a \textbf{continuous p-Schauder frame}  for $\mathcal{X}$   ($1<p<\infty$) if the following holds. 	
	\begin{enumerate}[\upshape(i)]
		\item For every $x\in \mathcal{X}$, the map $\Omega \ni \alpha \mapsto f_\alpha(x)\in \mathbb{K}$ 	is measurable.
		\item For every $x \in \mathcal{X}$, 
		\begin{align*}
			\|x\|^p=\int\limits_{\Omega}|f_\alpha(x)|^p\, d \mu(\alpha).
		\end{align*}
		\item For every $x\in \mathcal{X}$, the map 
		$
			\Omega \ni \alpha \mapsto f_\alpha(x)\tau_\alpha \in \mathcal{X}$
		is weakly measurable.
		 	\item For every $x \in \mathcal{X}$, 
		\begin{align*}
			x=\int\limits_{\Omega}	f_\alpha (x)\tau_\alpha \, d \mu(\alpha),
		\end{align*}  
	where the 	integral is weak integral.
	\end{enumerate}
\end{definition}
As an example,   all group-frames for Banach spaces are continuous p-Schauder frames (see \cite{KRISHNA4}). We can produce more examples using Hilbert space continuous frames.
\begin{example}
Let  $ \{\tau_\alpha\}_{\alpha\in \Omega}$   be a Parseval continuous frame for a Hilbert space $\mathcal{H}$. Let    $ \{\omega_\alpha\}_{\alpha\in \Omega}$ 	   be a continuous frame  $\mathcal{H}$ which is a dual of $ \{\tau_\alpha\}_{\alpha\in \Omega}$. Note that $ \{\omega_\alpha\}_{\alpha\in \Omega}$ need not be Parseval. Then $(\{\omega_\alpha\}_{\alpha\in \Omega}, \{\tau_\alpha\}_{\alpha\in \Omega})$   is a  continuous 2-Schauder frame  for $\mathcal{H}$.
\end{example}
We note that condition (i) in Definition \ref{CPSF} says that the map 
\begin{align*}
\theta_f: \mathcal{X} \ni x \mapsto \theta_fx \in \mathcal{L}^p(\Omega, \mu); \quad   \theta_fx: \Omega \ni \alpha \mapsto  (\theta_fx) (\alpha)\coloneqq f_\alpha (x) \in \mathbb{K}
\end{align*}
is a linear isometry. 
Following is the fundamental result of this paper. 
\begin{theorem}(\textbf{Functional Continuous  Uncertainty Principle})\label{CFDS}
Let $(\Omega, \mu)$,  $(\Delta, \nu)$ be   measure spaces. Let  $(\{f_\alpha\}_{\alpha\in \Omega}, \{\tau_\alpha\}_{\alpha\in \Omega})$  and   $(\{g_\beta\}_{\beta\in \Delta}, \{\omega_\beta\}_{\beta\in \Delta})$   be	continuous p-Schauder frames  for a Banach space $\mathcal{X}$. Then for every $x \in \mathcal{X}\setminus\{0\}$,  we have 
\begin{align}\label{CUP}
		\mu(\operatorname{supp}(\theta_f x))^\frac{1}{p}	\nu(\operatorname{supp}(\theta_g x))^\frac{1}{q} \geq 	\frac{1}{\displaystyle\sup_{\alpha \in \Omega, \beta \in \Delta}|f_\alpha(\omega_\beta)|}, \quad  	\nu(\operatorname{supp}(\theta_g x))^\frac{1}{p}	\mu(\operatorname{supp}(\theta_f x))^\frac{1}{q}\geq \frac{1}{\displaystyle\sup_{\alpha \in \Omega , \beta \in \Delta}|g_\beta(\tau_\alpha)|}.
\end{align}
\end{theorem}
\begin{proof}
	Let $x \in \mathcal{X}\setminus\{0\}$ and $q$ be the conjugate index of $p$. First using $\theta_f$ is an isometry and later using $\theta_g$ is an isometry, we get 
	
	\begin{align*}
	\|x\|^p&=\|\theta_fx\|^p=\int\limits_{\Omega}|f_\alpha (x)|^p\, d\mu(\alpha)=\int\limits_{ \operatorname{supp}(\theta_fx)}|f_\alpha(x)|^p\, d\mu(\alpha)\\
	&=\int\limits_{\operatorname{supp}(\theta_fx)}\left|f_\alpha\left(\int\limits_{\Delta}g_\beta(x)\omega_\beta\, d\nu(\beta)\right)\right|^p\, d\mu(\alpha)=\int\limits_{\operatorname{supp}(\theta_fx)}\left|\int\limits_{\Delta}g_\beta(x)f_\alpha(\omega_\beta)\, d\nu(\beta)\right|^p\, d\mu(\alpha)\\
	&=\int\limits_{\operatorname{supp}(\theta_fx)}\left|\int\limits_{ \operatorname{supp}(\theta_gx)}g_\beta(x)f_\alpha(\omega_\beta)\, d\nu(\beta)\right|^p\, d\mu(\alpha)\leq \int\limits_{ \operatorname{supp}(\theta_fx)}\left(\int\limits_{ \operatorname{supp}(\theta_gx)}|g_\beta(x)f_\alpha(\omega_\beta)|\, d\nu(\beta)\right)^p\, d\mu(\alpha)\\
	&\leq \left(\displaystyle\sup_{\alpha \in \Omega , \beta \in \Delta}|f_\alpha(\omega_\beta)|\right)^p\int\limits_{ \operatorname{supp}(\theta_fx)}\left(\int\limits_{\operatorname{supp}(\theta_gx)}|g_\beta(x)|\, d\nu(\beta)\right)^p\, d\mu(\alpha)\\
	&=\left(\displaystyle\sup_{\alpha \in \Omega , \beta \in \Delta}|f_\alpha(\omega_\beta)|\right)^p 	\mu(\operatorname{supp}(\theta_f x))\left(\int\limits_{\operatorname{supp}(\theta_gx)}|g_\beta(x)|\, d\nu(\beta)\right)^p\\
	&\leq \left(\displaystyle\sup_{\alpha \in \Omega , \beta \in \Delta}|f_\alpha(\omega_\beta)|\right)^p 	\mu(\operatorname{supp}(\theta_f x))\left(\int\limits_{\operatorname{supp}(\theta_gx)}|g_\beta(x)|^p\, d\nu(\beta)\right)^\frac{p}{p}\left(\int\limits_{ \operatorname{supp}(\theta_gx)}1^q\, d\nu(\beta)\right)^\frac{p}{q}\\
	&=\left(\displaystyle\sup_{\alpha \in \Omega , \beta \in \Delta}|f_\alpha(\omega_\beta)|\right)^p	\mu(\operatorname{supp}(\theta_f x))\|\theta_gx\|^p	\nu(\operatorname{supp}(\theta_g x))^\frac{p}{q}\\
	&=\left(\displaystyle\sup_{\alpha \in \Omega , \beta \in \Delta}|f_\alpha(\omega_\beta)|\right)^p 	\mu(\operatorname{supp}(\theta_f x))\|x\|^p\nu(\operatorname{supp}(\theta_g x))^\frac{p}{q}.	
	\end{align*}
Therefore 
\begin{align*}
	\frac{1}{\displaystyle\sup_{\alpha \in \Omega , \beta \in \Delta}|f_\alpha(\omega_\beta)|}\leq 	\mu(\operatorname{supp}(\theta_f x))^\frac{1}{p}\nu(\operatorname{supp}(\theta_g x))^\frac{1}{q}.
\end{align*}
On the other way, first using $\theta_g$ is an isometry and $\theta_f$ is an isometry, we get

\begin{align*}
	\|x\|^p&=\|\theta_gx\|^p=\int\limits_{\Delta}|g_\beta(x)|^p\, d\nu(\beta)=\int\limits_{ \operatorname{supp}(\theta_gx)}|g_\beta(x)|^p\, d\nu(\beta)\\
	&=\int\limits_{\operatorname{supp}(\theta_gx)}\left|g_\beta\left(\int\limits_{\Omega}f_\alpha(x)\tau_\alpha\, d\mu(\alpha)\right)\right|^p\, d\nu(\beta)
	=\int\limits_{\operatorname{supp}(\theta_gx)}\left|\int\limits_{\Omega}f_\alpha(x)g_\beta(\tau_\alpha)\, d\mu(\alpha)\right|^p\, d\nu(\beta)\\
	&=\int\limits_{\operatorname{supp}(\theta_gx)}\left|\int\limits_{ \operatorname{supp}(\theta_fx)}f_\alpha(x)g_\beta(\tau_\alpha)\, d\mu(\alpha)\right|^p\, d\nu(\beta)\leq \int\limits_{ \operatorname{supp}(\theta_gx)}\left(\int\limits_{ \operatorname{supp}(\theta_fx)}|f_\alpha(x)g_\beta(\tau_\alpha)|\, d\mu(\alpha)\right)^p\, d\nu(\beta)\\
	&\leq \left(\displaystyle\sup_{\alpha \in \Omega , \beta \in \Delta}|g_\beta(\tau_\alpha)|\right)^p\int\limits_{ \operatorname{supp}(\theta_gx)}\left(\int\limits_{ \operatorname{supp}(\theta_fx)}|f_\alpha(x)|\, d\mu(\alpha)\right)^p\, d\nu(\beta)\\
	&=\left(\displaystyle\sup_{\alpha \in \Omega , \beta \in \Delta}|g_\beta(\tau_\alpha)|\right)^p\nu(\operatorname{supp}(\theta_g x))\left(\int\limits_{ \operatorname{supp}(\theta_fx)}|f_\alpha(x)|\, d\mu(\alpha)\right)^p\\
	&\leq \left(\displaystyle\sup_{\alpha \in \Omega , \beta \in \Delta}|g_\beta(\tau_\alpha)|\right)^p\nu(\operatorname{supp}(\theta_g x))\left(\int\limits_{ \operatorname{supp}(\theta_fx)}|f_\alpha(x)|^p\, d\mu(\alpha)\right)^\frac{p}{p}\left(\int\limits_{ \operatorname{supp}(\theta_fx)}1^q\, d\mu(\alpha)\right)^\frac{p}{q}\\
	&=\left(\displaystyle\sup_{\alpha \in \Omega , \beta \in \Delta}|g_\beta(\tau_\alpha)|\right)^p\nu(\operatorname{supp}(\theta_g x))\|\theta_f x\|^p\mu(\operatorname{supp}(\theta_f x))^\frac{p}{q}\\
	&=\left(\displaystyle\sup_{\alpha \in \Omega , \beta \in \Delta}|g_\beta(\tau_\alpha)|\right)^p\nu(\operatorname{supp}(\theta_g x))\|x\|^p\mu(\operatorname{supp}(\theta_f x))^\frac{p}{q}.
\end{align*}
Therefore 
\begin{align*}
	\frac{1}{\displaystyle\sup_{\alpha \in \Omega , \beta \in \Delta}|g_\beta(\tau_\alpha)|}\leq \nu(\operatorname{supp}(\theta_g x))^\frac{1}{p}\mu(\operatorname{supp}(\theta_f x))^\frac{1}{q}.
\end{align*}	
\end{proof}
\begin{corollary}
Let $(\Omega, \mu)$,  $(\Delta, \nu)$ be   measure spaces. Let  $ \{\tau_\alpha\}_{\alpha\in \Omega}$  and   $ \{\omega_\beta\}_{\beta\in \Delta}$ 	   be Parseval continuous frames for a Hilbert space $\mathcal{H}$.  Then for every $h \in \mathcal{H}\setminus\{0\}$,  we have
\begin{align*}
	\mu(\operatorname{supp}(\theta_\tau h))	\nu(\operatorname{supp}(\theta_\omega h)) \geq 	\frac{1}{\displaystyle\sup_{\alpha \in \Omega, \beta \in \Delta}|\langle \omega_\beta, \tau_\alpha \rangle |^2}, \quad  	\nu(\operatorname{supp}(\theta_\omega h))	\mu(\operatorname{supp}(\theta_\tau  h))\geq \frac{1}{\displaystyle\sup_{\alpha \in \Omega , \beta \in \Delta}|\langle \tau_\alpha, \omega_\beta \rangle |^2},	
\end{align*}
where 
\begin{align*}
	&\theta_\tau: \mathcal{H} \ni h \mapsto \theta_\tau h \in \mathcal{L}^2(\Omega, \mu); \quad   \theta_\tau h: \Omega \ni \alpha \mapsto  (\theta_\tau h) (\alpha)\coloneqq \langle h, \tau _\alpha\rangle  \in \mathbb{K},\\
	&\theta_\omega: \mathcal{H} \ni h \mapsto \theta_\omega h \in \mathcal{L}^2(\Delta, \nu); \quad   \theta_\omega h: \Delta \ni \beta \mapsto  (\theta_\omega h) (\beta)\coloneqq \langle h, \omega_\beta \rangle  \in \mathbb{K}.
\end{align*}
\end{corollary}
\begin{corollary}
	Theorem \ref{FDS} follows from Theorem \ref{CFDS}.	
\end{corollary}
\begin{proof}
	Let $(\{f_j\}_{j=1}^n, \{\tau_j\}_{j=1}^n)$ and $(\{g_k\}_{k=1}^m, \{\omega_k\}_{k=1}^m)$ be  p-Schauder frames  for a finite dimensional Banach space $\mathcal{X}$. Define	$\Omega\coloneqq \{1, \dots, n\}$ and  $\Delta\coloneqq \{1, \dots, m\}$.  Take $\mu$ as counting measure on $\Omega$ and $\nu$ as counting measure on $\Delta$.
\end{proof}

Theorem  \ref{CFDS}  brings the following question.
\begin{question}
	\begin{enumerate}[\upshape(i)]
		\item 
	 Let $(\Omega, \mu)$,  $(\Delta, \nu)$ be   measure spaces,    $\mathcal{X}$ be a Banach space and $p>1$. For which pairs of continuous p-Schauder frames $(\{f_\alpha\}_{\alpha\in \Omega}, \{\tau_\alpha\}_{\alpha\in \Omega})$  and   $(\{g_\beta\}_{\beta\in \Delta}, \{\omega_\beta\}_{\beta\in \Delta})$ for  $\mathcal{X}$, we have equality in Inequality (\ref{CUP})?
	 \item What is the version of Theorem \ref{CFDS} for $0<p\leq1$ and $p=\infty$?
	 \item If $\mathcal{X}$ is  a Banach space of prime dimension, whether we can improve Theorem  \ref{CFDS} like Terence Tao's  \cite{TAO} (cf. \cite{MURTY}) improvement of Donoho-Stark uncertainty principle? More generally, whether we have  Roy  \cite{MESHULAM} and Murty-Whang theorems \cite{MURTYWHANG} for finite dimensional Banach spaces?
	 	\end{enumerate}
\end{question}
If we can derive the uncertainty principles derived in \cite{KRISHNA3} and  \cite{KRISHNA2} for p-Schauder frames, then we hope that we can get continuous versions of them.

\section{Measure minimization conjecture}
Uncertainty principles have connections with solution of sparse representation problem. The following result is the most important.
\begin{theorem} \cite{DONOHOELAD}  \label{DONOHOELAD} \textbf{(Donoho-Elad Sparsity Theorem)}
	Let $\{\omega_j\}_{j=1}^n$ be a  frame   for a  finite dimensional Hilbert space $\mathcal{H}$. If $h\in\mathcal{H}$ can be written as 
	\begin{align*}
		h=\theta_\tau^*c \quad \text{for some} \quad c\in \mathbb{K}^n \quad \text{with} \quad \|c\|_0<\frac{1}{2}\left(1+\frac{1}{\displaystyle\max_{1\leq j, k \leq n, j\neq k}|\langle\tau_j, \tau_k \rangle|}\right),
	\end{align*} 
then $c$ is the unique solution of the $\ell_0$-minimization problem 
\begin{align*}
(\text{P}_0) \quad \quad	\operatorname{minimize}\{\|d\|_0:d\in \mathbb{K}^n\} \quad \text{subject to} \quad h=\theta_\tau^*d.
\end{align*}
\end{theorem}
Based   on  Theorem \ref{DONOHOELAD} and  our uncertainty principle (Theorem \ref{CFDS}) we make the following conjecture.
\begin{conjecture}
\textbf{(Measure Minimization Conjecture)		Let $\{\tau_\alpha\}_{\alpha\in \Omega}$ be a  continuous frame   for a   Hilbert space $\mathcal{H}$. If $h\in\mathcal{H}$ can be written as 
\begin{align*}
	h=\theta_\tau^*f \quad \text{for some} \quad f\in \mathcal{L}^2(\Omega, \mu) \quad \text{with} \quad \mu(\operatorname{supp}(f))<\frac{1}{2}\left(1+\frac{1}{\displaystyle\sup_{\alpha, \beta  \in \Omega , \alpha \neq \beta}|\langle \tau_\alpha, \tau_\beta \rangle |}\right),
\end{align*} 
then $f$ is the unique solution of the measure minimization problem 
\begin{align*}
(\text{P}_\text{M}) \quad \quad 	\operatorname{minimize}\{\mu(\operatorname{supp}(g)):g\in \mathcal{L}^2(\Omega, \mu)\} \quad \text{subject to} \quad h=\theta_\tau^*g.
\end{align*}}
\end{conjecture}
It is known that Problem (P$_0$) is NP-Hard \cite{FOUCARTRAUHUT, NATARAJAN}. We therefore have the following problem.
\begin{problem}
\textbf{For which measure spaces $(\Omega, \mu)$, the Problem ($\text{P}_\text{M}$) is NP-Hard?}
\end{problem}

\section{Acknowledgments}
Author thanks Prof. Philip B. Stark, Department of Statistics, University of California, Berkeley, USA for asking Question \ref{SP}. Author believes that this paper exists because of Question \ref{SP}.

 \bibliographystyle{plain}
 \bibliography{reference.bib}

\begin{thebibliography}{10}

\bibitem{DONOHOELAD}
David~L. Donoho and Michael Elad.
\newblock Optimally sparse representation in general (nonorthogonal)
  dictionaries via {$l^1$} minimization.
\newblock {\em Proc. Natl. Acad. Sci. USA}, 100(5):2197--2202, 2003.

\bibitem{DONOHOSTARK}
David~L. Donoho and Philip~B. Stark.
\newblock Uncertainty principles and signal recovery.
\newblock {\em SIAM J. Appl. Math.}, 49(3):906--931, 1989.

\bibitem{ELADBRUCKSTEIN}
Michael Elad and Alfred~M. Bruckstein.
\newblock A generalized uncertainty principle and sparse representation in
  pairs of bases.
\newblock {\em IEEE Trans. Inform. Theory}, 48(9):2558--2567, 2002.

\bibitem{FOUCARTRAUHUT}
Simon Foucart and Holger Rauhut.
\newblock {\em A mathematical introduction to compressive sensing}.
\newblock Applied and Numerical Harmonic Analysis. Birkh\"{a}user/Springer, New
  York, 2013.

\bibitem{KRISHNA3}
K.~Mahesh Krishna.
\newblock Functional {D}onoho-{S}tark approximate-support uncertainty
  principle.
\newblock {\em arXiv:2307.01215v1 [math.FA] 1 July}, 2023.

\bibitem{KRISHNA1}
K.~Mahesh Krishna.
\newblock Functional {D}onoho-{S}tark-{E}lad-{B}ruckstein-{R}icaud-{T}orrésani
  uncertainty principle.
\newblock {\em arXiv: 2304.03324v1 [math.FA] 5 April}, 2023.

\bibitem{KRISHNA2}
K.~Mahesh Krishna.
\newblock Functional {G}hobber-{J}aming uncertainty principle.
\newblock {\em arXiv:2306.01014v1 [math.FA] 1 June}, 2023.

\bibitem{KRISHNA4}
K.~Mahesh Krishna.
\newblock Group-frames for {B}anach spaces.
\newblock {\em arXiv:2305.01499v1 [math.FA] 1 May}, 2023.

\bibitem{MESHULAM}
Roy Meshulam.
\newblock An uncertainty inequality for finite abelian groups.
\newblock {\em European J. Combin.}, 27(1):63--67, 2006.

\bibitem{MURTYWHANG}
M.~Ram Murty and Junho~Peter Whang.
\newblock The uncertainty principle and a generalization of a theorem of {T}ao.
\newblock {\em Linear Algebra Appl.}, 437(1):214--220, 2012.

\bibitem{NATARAJAN}
B.~K. Natarajan.
\newblock Sparse approximate solutions to linear systems.
\newblock {\em SIAM J. Comput.}, 24(2):227--234, 1995.

\bibitem{MURTY}
M.~Ram~Murty.
\newblock Some remarks on the discrete uncertainty principle.
\newblock In {\em Highly composite: papers in number theory}, volume~23 of {\em
  Ramanujan Math. Soc. Lect. Notes Ser.}, pages 77--85. Ramanujan Math. Soc.,
  Mysore, 2016.

\bibitem{RICAUDTORRESANI}
Benjamin Ricaud and Bruno Torr\'{e}sani.
\newblock Refined support and entropic uncertainty inequalities.
\newblock {\em IEEE Trans. Inform. Theory}, 59(7):4272--4279, 2013.

\bibitem{TALAGRAND}
Michel Talagrand.
\newblock Pettis integral and measure theory.
\newblock {\em Mem. Amer. Math. Soc.}, 51(307):ix+224, 1984.

\bibitem{TAO}
Terence Tao.
\newblock An uncertainty principle for cyclic groups of prime order.
\newblock {\em Math. Res. Lett.}, 12(1):121--127, 2005.

\end{thebibliography}

\end{document}